\documentclass[10pt]{amsart}
\usepackage{amsmath, amssymb, euscript,  enumerate}
\usepackage{color}
\usepackage{epsfig,graphicx}
\usepackage{ulem}
\usepackage{cite}
\usepackage{pxfonts}
\usepackage{color,wrapfig}
\usepackage[all]{xy}
\usepackage{amscd}
\usepackage{hyperref}
\usepackage{url}

\newtheorem{theorem}{Theorem}
\newtheorem{lemma}[theorem]{Lemma}
\newtheorem{corollary}[theorem]{Corollary}
\theoremstyle{definition}
\newtheorem{definition}[theorem]{Definition}
\newtheorem{example}[theorem]{Example}

\theoremstyle{remark}
\newtheorem{remark}[theorem]{Remark}
\numberwithin{equation}{section}

\DeclareMathOperator{\Long}{Length}
\DeclareMathOperator{\Area}{Area}
\newcommand{\df}{\mathrm{d}}

\begin{document}
\title[Generalized Calabi correspondence and complete spacelike surfaces]{
Generalized Calabi correspondence and \\complete spacelike surfaces}

\author{Hojoo Lee}
\address{Hojoo Lee, Korea Institute for Advanced Study, Hoegiro 87, Dongdaemun-gu, Seoul, 02455, Korea.}
\email{momentmaplee@gmail.com}
 
\author{Jos\'{e} M. Manzano}
\address{Jos\'{e} M. Manzano, Department of Mathematics, King's College London, Strand WC2R 2LS, United Kingdom.}
\email{manzanoprego@gmail.com}

\thanks{2010 MSC: Primary 49Q05, 53A10; Secondary 53C50, 35B08}

\begin{abstract}
We construct a twin correspondence between graphs with prescribed mean curvature in three-dimensional Riemannian Killing submersions and spacelike graphs with prescribed mean curvature in three-dimensional Lorentzian Killing submersions. Our duality extends the Calabi correspondence between minimal graphs in the Euclidean space ${\mathbb{R}}^{3}$ and maximal graphs in the Lorentz--Minkowski spacetime ${\mathbb{L}}^{3}$, by allowing arbitrary prescribed mean curvature and bundle curvature. For instance, we transform the prescribed mean curvature equation in ${\mathbb{L}}^{3}$ into the minimal surface equation in the generalized Heisenberg space with prescribed bundle curvature. We present several applications of the twin correspondence to the study of the moduli space of complete spacelike surfaces in certain Lorentzian spacetimes.
\end{abstract}

\maketitle

\section{Introduction} \label{intro}

Maximal submanifolds in Lorentzian manifolds are spacelike submanifolds with vanishing mean curvature, and arise naturally as critical points of the area functional within the class of spacelike submanifolds. They have played a significant role in geometric analysis, as shown, for instance, by Schoen and Yau's proof of the Positive Mass Theorem, or by the analysis of solutions of the Einstein--Yang--Mills equation. Existence, uniqueness, and regularity of maximal (and, more generally, constant mean curvature) submanifolds in different spacetimes have become important problems in both mathematical relativity and differential geometry~\cite{BS82}. The main goal of this paper is to prove that there is a large family of Lorentzian Killing three-manifolds with timelike unit Killing vector fields not admitting complete spacelike surfaces.
 
In 1970, Calabi \cite{Cal70} proved the remarkable result that the only entire maximal graphs defined over the whole $xy$-plane in the Lorentz--Minkowski space ${\mathbb{L}}^{3}=({\mathbb{R}}^{3},\df x^2 + \df y^2 -\df z^2)$ are the spacelike affine planes. This implies that spacelike planes are the only complete maximal surfaces in ${\mathbb{L}}^{3}$, which can be considered the Lorentzian counterpart of Bernstein's beautiful result that the only entire minimal graphs in the Euclidean space ${\mathbb{E}}^{3}=({\mathbb{R}}^{3},\df x^2 + \df y^2 +\df z^2)$ are the affine planes, though ${\mathbb{E}}^{3}$ admits many other complete non-planar minimal surfaces. 

Cheng and Yau \cite{CY76} extended Calabi's result to higher dimensions by proving that the only complete maximal hypersurfaces in Lorentz--Minkowski space ${\mathbb{L}}^{n+1}$ with $n+1 \geq 3$ are the spacelike hyperplanes. Nonetheless, Bombieri, De Giorgi and Giusti \cite{BGG69} disproved the Riemannian analogue by showing the existence of entire non-planar minimal hypersurfaces in the Euclidean space ${\mathbb{E}}^{n+1}$ with $n+1 \geq 9$.

During the last few decades, a considerable attention has been paid to constant mean curvature surfaces in more general Lorentzian spaces admitting Killing or conformally Killing timelike directions:
\begin{itemize}
	\item Fern\'{a}ndez and Mira constructed a large family of complete entire maximal surfaces in the static Robertson-Walker spacetime $\mathbb H^2\times \mathbb{R}$ endowed with the usual Lorentzian product metric~\cite[Theorem 3.1]{FM07} (notice that entire graphs are not necessarily complete, see~\cite[Examples 3.1 and 3.3]{Alb08}). On the contrary, Albujer and Al\'{i}as \cite[Theorem 3.3]{AA09} proved that, given an arbitrary Riemannian surface $M$ with Gaussian curvature $K_M\geq 0$, complete maximal surfaces in $M \times \mathbb{R}$ (endowed with the usual Lorentzian product metric) are totally geodesic. Furthermore, they show that if $K_M$ does not vanish identically, then $\Sigma$ must be a (spacelike) horizontal slice $M\times \{t_{0}\}$ for some $t_{0} \in \mathbb{R}$.
	\item Generalized Robertson-Walker (GWR) spacetimes also represent ambient spaces of increasing interest~\cite{Mon88, RR10b}. Given a smooth function $\phi:\mathbb{R} \rightarrow (0, +\infty)$ and a Riemannian manifold $M$, the GRW spacetime $M {\times}_{\phi} \mathbb{R}$ is defined as $M\times \mathbb{R}$ equipped with the Lorentzian warped metric ${\phi(z)}^{2}{\langle\cdot,\cdot\rangle}_M-\df z^{2}$. These are conformally stationary spacetimes, i.e., they are time-orientable spacetimes admitting a timelike conformal vector field $\phi(z)\,\partial_z$. For instance, see Montiel's observation \cite[Section 3]{Mon88}. Several existence results of spacelike immersions in conformally stationary spacetimes are also known \cite{JS08, Har92}.
\end{itemize}

Calabi's original proof \cite{Cal70} of his Bernstein type result relies on an interesting duality between minimal graphs in ${\mathbb{E}}^{3}$ and maximal graphs in ${\mathbb{L}}^{3}$. This twin correspondence extends to a correspondence between graphs of constant mean curvature $H$ in the Riemannian Bianchi--Cartan--Vranceanu (BCV) space ${\mathbb{E}}^{3}(\kappa, \tau)$ and spacelike graphs of constant mean curvature $\tau$ in the Lorentzian BCV space ${\mathbb{L}}^{3}(\kappa, H)$, see~\cite{Lee12} and Corollary~\ref{coro:bcv} below, as well as the explicit definition of $\mathbb{E}(\kappa,\tau)$ and $\mathbb{L}(\kappa, H)$ in Remark~\ref{BCV-twisted} below. The importance of these 2-parameter families in the literature comes from the fact that they model all homogeneous 3-manifolds with isometry group of dimension $4$, which are the most symmetric 3-manifolds after the space-forms. In the particular case $\tau=H=0$, the duality reduces to the Albujer--Al\'{i}as duality \cite{AA09}. It is worth mentioning that there are also extensions to higher codimension, see~\cite{Lee12}.

Since the twin correspondence in BCV spaces sends entire graphs to entire graphs, it becomes a natural and useful tool for studying Bernstein--Calabi type problems. For instance, entire spacelike graphs of constant mean curvature $\frac{1}{2}$ in ${\mathbb{L}}^{3}={\mathbb{L}}^{3}(0,0)$ correspond to entire minimal graphs in the Heisenberg space $\mathrm{Nil}^3=\mathbb{E}^{3}(0,\frac{1}{2})$. As the moduli space of entire spacelike graphs of constant mean curvature $\frac{1}{2}$ in $\mathbb{L}^3$ is large (Treibergs \cite[Theorem 2]{Tr82a} showed that the asymptotic behaviour of such graphs can be an arbitrary $\mathcal C^2$-perturbation of the light cone in $\mathbb{L}^3$), it follows that there exist many entire minimal graphs in $\mathrm{Nil}^3$. 
A description of the moduli space of entire minimal graphs in $\mathrm{Nil}^3$ was determined by Fern\'{a}ndez and Mira~\cite{FM07} in terms of holomorphic quadratic differentials.

In this paper, we generalize the Calabi correspondence to the case of three-manifolds admitting a unit Killing vector field, and, more specifically, to the case of unit Killing submersions over a non-compact simply connected Riemannian surface $M$. A submersion $\pi:\mathbb E\to M$ from an orientable Lorentzian or Riemannian three-manifold $\mathbb{E}$ is said Killing if it preserves the length of horizontal vectors and its fibers are the integral curves of a unit Killing vector field $\xi$ in $\mathbb{E}$~\cite{EO,LerMan}. The metric in $\mathbb E$ is Riemannian in the two-dimensional horizontal distribution orthogonal to $\xi$, but $\xi$ is a timelike vector field if $\mathbb E$ is assumed Lorentzian. 

There is a unique function $\tau\in\mathcal C^\infty(M)$ satisfying
\begin{equation}\label{eqn:bundle}
\overline\nabla_X\xi=\tau X\times\xi,\qquad \text{for all vector fields }X\text{ in } \mathbb E,
\end{equation}
where $\overline\nabla$ is the Levi-Civita connection in $\mathbb E$, and it is called the bundle curvature of $\pi$ (existence essentially follows from the fact that the curvature 2-form $\alpha(X,Y)=\langle\overline\nabla_X\xi,Y\rangle$ is skew-symmetric since $\xi$ is Killing). The bundle curvature accounts for the non-integrability of the horizontal distribution, in the sense that $\tau\equiv 0$ if and only if the horizontal distribution is integrable, see~\cite{EO,Man12b}. The cross product in~\eqref{eqn:bundle} is defined as $\langle Y, X\times \xi\rangle_{\mathbb E}=\det(Y,X,\xi)$ for all vector fields $X,Y$ in $\mathbb{E}$, where the determinant is computed with respect to the volume form determined by the orientation and metric in $\mathbb E$. It is important to note that (the sign of) the bundle curvature depends on the choice of the orientation of the total space.

Given a non-compact simply connected Riemannian surface $M$ and $\tau\in \mathcal{C}^\infty(M)$, there exists a unique Riemannian (resp. Lorentzian) Killing submersion over $M$ with bundle curvature $\tau$ whose fibers have infinite length. The Riemannian case follows from~\cite{Man12b} and the Lorentzian case is analogous. The total space of such a Killing submersion will be denoted by $\mathbb E^3(M,\tau)$ (resp. $\mathbb L^3(M,\tau)$). If $M$ is the simply connected surface with constant curvature $\kappa\in\mathbb R$ and the bundle curvature $\tau$ is also constant, then $\mathbb E^3(M,\tau)=\mathbb E^3(\kappa,\tau)$ and $\mathbb L^3(M,\tau)=\mathbb L^3(\kappa,\tau)$, so these three-manifolds can be interpreted as generalized Bianchi--Cartan--Vranceanu spaces. 

Explicit expressions for the metrics of $\mathbb E^3(M,\tau)$ and $\mathbb L^3(M,\tau)$ will be given in Section~\ref{GBCVs} in terms of $\tau$ and a conformal parametrization of $M$. Section~\ref{MTone} is devoted to establishing our Calabi type twin correspondence, which is involutive up to translations along the fibers, and preserves the conformal type (see Theorem~\ref{TCb} below). This correspondence applies to simply connected graphs (i.e., sections of the submersion) and swaps the mean and bundle curvatures: it relates graphs with prescribed mean curvature $H$ in $\mathbb{E}^{3}(M,\tau)$ and spacelike graphs with prescribed mean curvature $\tau$ in $\mathbb{L}^{3}(M,H)$. Notice that the base surface does not change in the correspondence. Its proof is based on Poincar\'{e} Lemma and the fact that both $\tau$ and $H$ admit divergence-form expressions (see Lemma~\ref{lemma:H} below). 

When the curvature of the base curvature, bundle curvature, and mean curvature are all constant, we recover the twin correspondence in the classical BCV spaces~\cite[Theorem 2]{Lee11a}. It is worth mentioning that Killing submersions do not admit warped-product structures in general, so our correspondence is not related to the correspondence constructed by G\'{a}lvez, Jim\'{e}nez and Mira \cite{JimGalMir} for isometric immersions in warped products.

In Section \ref{MTapp}, we will employ our duality to investigate complete spacelike surfaces in Lorentzian Killing submersions. Although completeness of entire spacelike graphs in Lorentzian manifolds is not guaranteed in general~\cite{Alb08}, we establish that complete spacelike surfaces in $\mathbb{L}^{3}(M,\tau)$ must be entire graphs. In particular, the study of complete spacelike surfaces in $\mathbb{L}^{3}(M,\tau)$ reduces to that of the entire spacelike solutions of the corresponding prescribed mean curvature equation, where classical PDE techniques can be applied. We remark that complete spacelike surfaces in $\mathbb{L}^{3}(M,\tau)$ are in correspondence with spacelike foliations via the one-parameter group of isometries associated with the vertical Killing vector field.

We prove that there is a large family of Lorentzian Killing submersions not admitting complete spacelike surfaces. This is achieved by means of a sharp bound on the mean curvature function of an entire graph in the Riemannian case.  More explicitly, we prove that if the bundle curvature $\tau$ of $\mathbb{L}^{3}(M,\tau)$ satisfies $\inf_M |\tau| >\tfrac{1}{2}\mathrm{Ch}(M)$, where $\mathrm{Ch}(M)$ denotes the Cheeger constant of the non-compact simply connected base $M$ (see Definition~\ref{defi:cheeger} below), then $\mathbb{L}^{3}(M,\tau)$ admits no complete spacelike surfaces. By using a comparison result for Cheeger constants, we get that, if the infimum of the Gaussian curvature of $M$ is a real number $c\leq 0$, and $\inf_M |\tau| >\tfrac{1}{2}\sqrt{-c}$, then $\mathbb{L}^{3}(M,\tau)$ does not admit complete spacelike surfaces, either. It is important to remark that there is no assumption on the mean curvature of the complete spacelike surface.
 
As a particular case, we deduce that there do not exist complete spacelike surfaces in $\mathbb{L}^{3}\left(\kappa,\tau\right)$ with $\kappa\leq0$ and $|\tau|>\frac{1}{2}\sqrt{-\kappa}$. Moreover, we show that the lower bound $\frac{1}{2}\sqrt{-\kappa}$ is sharp, for the moduli space of complete maximal surfaces in the anti--de Sitter spacetime $\mathbb{L}^{3}(\kappa,\frac{1}{2}\sqrt{-\kappa})$ is large.  In fact, we prove that a maximal surface in the anti--de Sitter spacetime $\mathbb{L}^{3}(\kappa,\frac{1}{2}\sqrt{-\kappa})$ is complete if and only if it is an entire spacelike graph defined over the whole hyperbolic base ${\mathbb{H}}^{2}(\kappa)$. This is an extension of Cheng and Yau's result~\cite{CY76} that entire constant mean curvature spacelike graphs in ${\mathbb{L}}^{3}=\mathbb L^3(0,0)$ are complete.

To conclude, we interpret our non-existence result in the relativistic sense of causality by proving that $\mathbb{L}^{3}(M,\tau)$ is not a \textit{distinguishing} spacetime provided that $\inf_M |\tau| >\tfrac{1}{2}\mathrm{Ch}(M)$.

\section{Preliminaries on Killing submersions} \label{GBCVs}

We will begin by defining models for the Killing submersions in terms of the bundle curvature. Given a non-compact and simply connected surface $M$, we can conformally parametrize $M=(\Omega,\delta^{-2}(x,y)(\df x^2+\df y^2))$, where $\Omega\subset\mathbb R^2$ is a disk or the whole plane, and $\delta\in\mathcal C^\infty(\Omega)$ is some positive function. Since the definitions and results below make sense in the more general case $\Omega$ is star-shaped with respect to the origin, we will assume this condition in the sequel.

\begin{definition}[\textbf{Calabi potential}]
Given $\tau\in\mathcal C^2(\Omega)$, we will call \textit{Calabi potential} the function $\mathbf{C}_{\delta,\tau}\in\mathcal C^2(\Omega)$ defined by 
\begin{equation} \label{Cpotential01}
 \mathbf{C}_{\delta,\tau}(x,y)
  = 2\int_{0}^{1}  t\, \frac{ \tau( t x, t y) }
  { {\delta(t x,t y)}^{2}} \, dt,
\end{equation}
\end{definition}

This definition is inspired by the explicit metrics given in~\cite{Man12b} in the Riemannian setting, which easily extend to the Lorentzian case. By taking derivatives under the integral sign, it is easy to check that $\mathbf{C}_{\delta,\tau}$ fulfils the divergence-form equation
\begin{equation}\label{eqn:calabi-divergence}
  \frac{2\tau}{{\delta}^{2}} = \frac{\partial}{\partial x} \left(x\,{\mathbf{C}}_{{\delta, \tau}}   \right) + \frac{\partial}{\partial y}
  \left(  y\,{\mathbf{C}}_{{\delta, \tau}}   \right),
\end{equation}
which plays a fundamental role in the proof of the Calabi type correspondence. 

We define $\mathbb E^3(M,\tau)$ and $\mathbb L^3(M,\tau)$ as the three-dimensional product space $\Omega\times\mathbb R$ endowed with the metric
 \begin{equation} \label{Rmetric}
  \frac{1}{{\delta(x,y)}^{2}} (\df x^2 + \df y^2)
   + \epsilon\left( \df z + \epsilon\, \mathbf{C}_{\delta,\tau}(x,y) \left( y\,\df x - x\,\df y  \right) \,
   \right)^{2}.
 \end{equation}
with $\epsilon=1$ in the case of $\mathbb E^3(M,\tau)$ and $\epsilon=-1$ in the case of $\mathbb L^3(M,\tau)$. 

The natural projection $\pi:\Omega\times\mathbb R\to \Omega$ given by $\pi(x,y,z)=(x,y)$ is the only Killing submersion with bundle curvature $\tau$ over $M$ when the above metrics in $\Omega\times\mathbb R$ are considered. The Killing vector field $\xi$ is spacelike if and only if $\epsilon=1$. The proof in the Riemannian case follows from~\cite{Man12b}, and can be easily extended to the Lorentzian setting. Hence the Calabi potential can be thought as a way of recovering the metric of the total space in terms of its bundle curvature.

\begin{remark}[\textbf{Killing submersions as extensions of product and BCV spaces}]\label{BCV-twisted}
If the bundle curvature function $\tau$ vanishes identically, then $\mathbf{C}_{\delta,0}\equiv 0$. We get the spaces $\mathbb{E}^{3}(M,0)=M \times \mathbb{R}$ with the Riemannian product metric ${\langle\cdot,\cdot\rangle}_{M}+\df z^{2}$, and $\mathbb{L}^{3}(M,0)= M\times \mathbb{R}$ with the Lorentzian product metric ${\langle\cdot,\cdot\rangle}_M-\df z^{2}$. A more geometric interpretation of the bundle curvature $\tau$ can be found in \cite[Proposition~3.3]{Man12b}. 

On the other hand, classical BCV spaces have constant base curvature and constant bundle curvature. Given $\kappa\in\mathbb R$, let $\delta(x,y)=1+\frac{\kappa}{4}(x^2+y^2)$, which is positive in $\Omega=\mathbb R^2$ for $\kappa\geq 0$ or $\Omega=\{(x,y)\in\mathbb R^2:x^2+y^2<\frac{-4}{\kappa}\}$ for $\kappa<0$. Then the metric $g=\delta(x,y)^{-2} (\df x^2 + \df y^2)$ on $\Omega$ has constant curvature $\kappa$, and $M=(\Omega,g)$ is isometric to $\mathbb E^2$ (for $\kappa=0$), $\mathbb H^2(\kappa)$ (for $\kappa<0$), or $\mathbb S^2(\kappa)$ minus a point (for $\kappa>0$). When $\tau(x,y)$ is constant, the associated Calabi potential is nothing but $\mathbf{C}_{\delta,\tau}(x,y)=\tau\,\delta(x,y)^{-1}$, so we get the classical BCV spaces ${\mathbb {E}}^{3}(\kappa,\tau)$ and ${\mathbb{L}}^{3}(\kappa,\tau)$, see~\cite{Dan07}. We will use the classical notation ${\mathbb {E}}^{3}(\kappa,\tau)=\mathbb{E}^{3}({\mathbb{M}}^{2}(\kappa),\tau)$ and ${\mathbb {L}}^{3}(\kappa,\tau)=\mathbb{L}^{3}({\mathbb{M}}^{2}(\kappa),\tau)$, where  ${\mathbb{M}}^{2}(\kappa)$ denotes the aforesaid surface with constant curvature $\kappa$.
\end{remark}

Given $\epsilon\in\{-1,1\}$, a global orthonormal frame $\{E_1,E_2,E_3\}$ in $\Omega\times\mathbb R$ endowed with the metric~\eqref{eqn:calabi-divergence} is given by
\begin{equation}\label{eqn:frame}
   E_{1}= \delta  \left(  {\partial_x} - \epsilon\, y\, {\mathbf{C}}_{{\delta, \tau}}  {\partial_z}  \right),\qquad
  E_{2}= \delta  \left(  {\partial_y} +\epsilon\,x\, {\mathbf{C}}_{{\delta, \tau}}    {\partial_z}  \right),\qquad
 E_{3}=   {\partial_z}.
\end{equation}
This frame will be assumed positively oriented without loss of generality. We recall that an orientation of the total space is needed to define the bundle curvature.

\section{Twin correspondences} \label{MTone}

We will now derive formulas for the mean curvature of a vertical graph in $\mathbb{E}^3(M,\tau)$ and $\mathbb{L}^3(M,\tau)$. Using the notation of the previous section, the graph of a function $u\in\mathcal C^2(\Omega')$, where $\Omega'\subseteq\Omega$ is an open subset, is the surface 
\[\left\{(x,y,z)\in\Omega'\times\mathbb R:z=u(x,y)\right\}.\]
In other words, we consider graphs over the global zero-section $z=0$. The graph is said to be \textit{entire} when $\Omega'=\Omega$. In the Lorentzian case, the graph is said  to be \textit{spacelike} if the induced metric from the ambient manifold is Riemannian. 

The following lemma can be easily deduced.

\begin{lemma}[\textbf{Mean curvature of a graph}] \label{lemma:H}
Let $u\in\mathcal C^2(\Omega')$ denote the height function of a graph defined on some open subset $\Omega'\subseteq\Omega$. 
 \begin{enumerate}
 \item The mean curvature $H$ of the graph of $u$ in $\mathbb{E}^{3}(M, \tau)$ satisfies
\begin{equation}\label{eq:H}
2 H=\delta^2\left(\frac{\partial}{\partial
x}\left(\frac{\alpha}{\omega}\right)+\frac{\partial}{\partial
y}\left(\frac{\beta}{\omega}\right)\right)=\mathrm{div}_{M}\left(\frac{G}{\sqrt{1+{\|G\|}_{M}^2}}\right),
\end{equation}
where $\alpha=u_x+y\,\mathbf{C}_{\delta,\tau}$,
$\beta=u_y-x\,\mathbf{C}_{\delta,\tau}$,
$\omega=\sqrt{1+\delta^2(\alpha^2+\beta^2)}$, and the vector field $G$ on $\Omega'\subset M$ is given by $G=\delta^2(\alpha\partial_x+\beta \partial_y)$.

\item If the graph of $u$ is spacelike in $\mathbb{L}^{3}(M,\widetilde{\tau})$, then its mean curvature $\widetilde H$ satisfies
\begin{equation}\label{eq:H-lorentz}
2\widetilde H=\delta^2\left(\frac{\partial}{\partial
x}\left(\frac{\widetilde\alpha}{\widetilde\omega}\right)+\frac{\partial}{\partial
y}\left(\frac{\widetilde\beta}{\widetilde\omega}\right)\right)=\mathrm{div}_{M}\left(\frac{\widetilde{G}}{\sqrt{1- {\|\widetilde{G}\|}_M^2}}\right),
\end{equation}
where
$\widetilde\alpha=u_x-y\,\mathbf{C}_{\delta,\widetilde{\tau}}$,
$\widetilde\beta=u_y+x\,\mathbf{C}_{\delta,\widetilde{\tau}}$,
 $\widetilde{\omega} = \sqrt{ 1 - {{\delta}}^{2} ( {\widetilde{\alpha}}^2 + { \widetilde{\beta}}^2
 )}$, and the vector field $\widetilde{G}$ on $\Omega'\subset M$ is given by $\widetilde G=\delta^2(\widetilde{\alpha} {\partial_x}+\widetilde{\beta}{\partial_y})$. The graph of $u$ is spacelike in $\mathbb{L}^{3}(M,
  \widetilde{\tau})$ if and only if
\[1-{\|\widetilde{G}\|}_{M}^{2} = 1 - {{\delta}}^{2} ( {\widetilde{\alpha}}^2 + { \widetilde{\beta}}^2)>0.\]
 \end{enumerate}
\end{lemma}
 
\begin{example}[\textbf{Helicoids in spaces with rotational symmetry}] 
A vertical translation in a Killing submersion is an element of the $1$-parameter group of isometries associated with the unit Killing vector field. Unless there is an isometry of the base surface that leaves the bundle curvature invariant or changes its sign, vertical translations are the only isometries in the total space of the Killing submersion, see~\cite{Man12b}.

If assume that both the conformal factor and the bundle curvature are radial with respect to the origin (i.e., $\delta(x,y)$ and $\tau(x,y)$ are functions of $x^2+y^2$), then the induced Calabi potential ${\mathbf{C}}_{{\delta, \tau}}$ is also radial. Given $\mu_1,\mu_2\in\mathbb R$, let us consider (with respect to the metric given by~\eqref{Rmetric}) the \textit{helicoid} 
\[{\mathcal{H}}_{{\mu}_{1}, {\mu}_{2}}=\{(\rho\cos(\theta),\rho\sin(\theta),\mu_1\theta+\mu_2):\rho,\theta\in\mathbb{R}\}\]
Each level curve ${\mathcal{H}}_{{\mu}_{1}, {\mu}_{2}} \cap \{z=z_{0}\}$ is an ambient geodesic, and symmetries about these geodesics are ambient isometries. This implies that $\mathcal H_{\mu_1,\mu_2}$ is a \textit{ruled} surface with zero mean curvature provided that $\delta$ and $\tau$ are radially symmetric. Note that $\mathcal H_{\mu_1,\mu_2}$ is a horizontal plane for $\mu_1=0$ and converges to a vertical plane when $\mu_1\to\infty$. Except at the axis $x=y=0$, $\mathcal H_{\mu_1,\mu_2}$ can be expressed locally as the graph
\[z=\mu_1\arctan\left(\frac{y}{x}\right)+\mu_2.\]
It is worth mentioning that the only ruled minimal surfaces in the Heisenberg space 
$\mathrm{Nil}^3=\mathbb{E}^{3}(0, \frac{1}{2})$ are (subsets of) planes, helicoids, and hyperbolic paraboloids up to ambient isometries \cite[Theorem 2.3]{KKLSY}. It would be interesting to extend this characterization to Killing submersions with rotational symmetry.
\end{example}

We now state our Calabi type correspondence between graphs with prescribed mean curvature $H$ in the Riemannian space $\mathbb{E}^{3}(M,\tau)$ and spacelike graphs with prescribed mean curvature $\tau$ in the Lorentzian space $\mathbb{E}^{3}(M,H)$.

\begin{theorem} [\textbf{Twin correspondence}] \label{TCb} 
Let $M$ be an open domain $\Omega\subseteq\mathbb R^2$, star-shaped with respect to the origin, endowed with the metric $\delta^{-2}(\df x^2+\df y^2)$ for some positive $\delta\in\mathcal C^\infty(\Omega)$. Given $\tau,H\in\mathcal C^2(\Omega)$ and a simply connected open domain  $\Omega'\subseteq\Omega$, then:
\begin{itemize}
\item[\textbf{(a)}]  If the graph of a function $f\in\mathcal C^2(\Omega')$ in $\mathbb{E}^{3}(M, \tau)$ has mean curvature $H$, then there exists $g\in \mathcal C^2(\Omega')$ such that
 the graph of $g$ is spacelike in $\mathbb{L}^{3}(M, H)$ and has mean curvature $\tau$.

\item[\textbf{(b)}]  If the graph of a function $g\in\mathcal C^2(\Omega')$ in $\mathbb{L}^{3}(M, H)$ is spacelike and has mean curvature $\tau$, then there exists $f \in\mathcal C^2(\Omega')$ such that the graph of $f$ in $\mathbb{E}^{3}(M, \tau)$ has mean curvature $H$.
\end{itemize}
The height functions of graphs in \textbf{(a)}  and \textbf{(b)} can be chosen to satisfy the twin relations:
\begin{equation}\label{twin}
\left( \widetilde{\alpha} , \widetilde{\beta} \right) = \left( - \frac{
\beta }{ \omega }, \frac{ \alpha }{ \omega} \right), \quad\text{or
equivalently,} \quad \left( \alpha , \beta \right) = \left( \frac{
\widetilde{\beta} }{ \widetilde{\omega} }, - \frac{ \widetilde{\alpha}
}{\widetilde{\omega} } \right),
\end{equation}
where
\begin{align*}
\left( \alpha, \beta \right)&= \left(
 f_{x} + y\, {\mathbf{C}}_{{\delta, \tau}}, f_{y} - x\, {\mathbf{C}}_{{\delta, \tau}} \right), &
 \omega&= \sqrt{ 1+ {{\delta}}^{2} \bigl( {\alpha}^2 + {\beta}^2 \bigr)},
\\
\bigl( \widetilde{\alpha}, \widetilde{\beta} \bigr) &= \left(
 g_{x} - y\, {\mathbf{C}}_{{\delta, H}}, g_{y} + x\, {\mathbf{C}}_{{\delta, H}} \right), &
 \widetilde{\omega}&= \sqrt{ 1 - {{\delta}}^{2} \bigl( {\widetilde{\alpha}}^2 + { \widetilde{\beta}}^2 \bigr)}.
\end{align*}
If we parametrize the graphs as $F(x,y)=(x,y,f(x,y))$ and $G(x,y)=(x,y,g(x,y))$, $(x,y)\in\Omega'$, then the map $F(x,y)\mapsto G(x,y)$ is a conformal diffeomorphism. More explicitly, if ${\mathbf{I}}_{F}$  and $\widetilde{\mathbf{I}}_{G}$ denote the induced metrics on the graphs $z=f(x,y)$ in $\mathbb E^3(M,\tau)$ and $z=g(x,y)$ in $\mathbb L^3(M,H)$, respectively, then
\begin{equation}
 { \widetilde{\mathbf{I}} }_{G}= \frac{1}{{\omega}^2}\,{\mathbf{I}}_{F}, \;\; \text{or
equivalently,}  \;\; {{\mathbf{I}} }_{F}= \frac{1}{{ \widetilde{\omega}}^2}\,{ \widetilde{\mathbf{I}} }_{G}, 
 \;\; \text{or symmetrically,}  \;\;  \frac{1}{{\omega}}\,{\mathbf{I}}_{F} = \frac{1}{{ \widetilde{\omega}}}\,{ \widetilde{\mathbf{I}} }_{G}.
\end{equation}
\end{theorem}

In the above conditions, the graph of $f$ in $\mathbb{E}^{3}(M,\tau)$ and the graph of $g$ in $\mathbb{L}^{3}(M, H)$ are called \textit{twin surfaces}.

\begin{proof} We will begin by proving \textbf{(a)}. The proof of \textbf{(b)} will be analogous.  Lemma \ref{lemma:H} shows that the mean curvature $H$ of the graph of $f$ over $\Omega'$ in $\mathbb{E}^{3}(M,\tau)$ satisfies the divergence-form equation
\begin{equation} \label{mean1}
\frac{2H}{{{\delta}}^{2}} = \frac{\partial}{\partial x} \left(
\frac{\alpha}{\omega} \right) + \frac{\partial}{\partial y} \left(
\frac{\beta}{\omega} \right),
\end{equation}
 where $( \alpha, \beta ) = ( f_{x} + y\, {\mathbf{C}}_{{\delta, \tau}}, f_{y} - x\,{\mathbf{C}}_{{\delta, \tau}} )$ and $\omega= \sqrt{ 1+ {{\delta}}^{2} \bigl( {\alpha}^2 + {\beta}^2\bigr)}$. Equation~\eqref{eqn:calabi-divergence} tells us that the Calabi potential ${\mathbf{C}}_{{\delta, \tau}}(x,y)$ satisfies the divergence-form equation
\begin{equation} \label{mean4}
  \frac{2\tau}{{\delta}^{2}} = \frac{\partial}{\partial x} \left(x\,{\mathbf{C}}_{{\delta, \tau}}   \right) + \frac{\partial}{\partial y}
  \left(  y\,{\mathbf{C}}_{{\delta, \tau}}   \right).
\end{equation}
  Likewise, the Calabi potential ${\mathbf{C}}_{{\delta, H}}(x,y)$ satisfies
\begin{equation} \label{mean2}
 \frac{2H}{{\delta}^{2}} =  \frac{\partial}{\partial x} \left(x\, {\mathbf{C}}_{{\delta, H}}  \right) +\frac{\partial}{\partial y} \left(  y\, {\mathbf{C}}_{{\delta, H}}   \right).
\end{equation}
 Using \eqref{mean2}, the equation \eqref{mean1} can be rewritten as the zero-divergence equation
\begin{equation*} \label{mean3}
 0 = \frac{\partial}{\partial x} \left( \frac{\alpha}{\omega} -  x\, {\mathbf{C}}_{{\delta, H}}   \right) + \frac{\partial}{\partial y}
 \left( \frac{\beta}{\omega} -  y\, {\mathbf{C}}_{{\delta, H}} \right).
\end{equation*}
 Since the domain $\Omega'$ is simply connected, Poincar\'{e} Lemma yields the existence of $g\in\mathcal C^2(\Omega')$ such that
\[
 \left( g_{x}, g_{y} \right) = \left( -\frac{\beta}{\omega} + y\,{\mathbf{C}}_{{\delta, H}},
 \frac{\alpha}{\omega} - x\,{\mathbf{C}}_{{\delta, H}}\right).
\]
 Setting $( \widetilde{\alpha}, \widetilde{\beta} ) = ( g_{x}- y\,{\mathbf{C}}_{{\delta, H}},
 g_{y} + x\, {\mathbf{C}}_{{\delta, H}} )$, we obtain the spacelike condition
\[
  1 - {{\delta}}^{2} \left( {\widetilde{\alpha}}^2 + { \widetilde{\beta}}^2 \right) =
  \frac{1}{{\omega}^{2}} > 0,
\]
 so it makes sense to introduce $\widetilde{\omega} = \sqrt{ 1 - {{\delta}}^{2} \bigl( {\widetilde{\alpha}}^2 + { \widetilde{\beta}}^2 \bigr)}$, from which the twin relations~\eqref{twin} easily follow. Finally, we employ the twin relations, the integrability condition $\frac{\partial}{\partial x}( f_{y})
  = \frac{\partial}{\partial y}( f_{x})$, and \eqref{mean4} to deduce 
\begin{align*}
 \frac{\partial}{\partial x} \left( \frac{ \widetilde{\alpha}}{ \widetilde{\omega}} \right) + \frac{\partial}{\partial y} \left( \frac{\widetilde{\beta}}{ \widetilde{\omega} } \right)
  &=   \frac{\partial}{\partial x} \left(  -\beta   \right) + \frac{\partial}{\partial y} \left( \alpha  \right) \\
  &=   \frac{\partial}{\partial x} \left( -  f_{y} + x\, {\mathbf{C}}_{{\delta, \tau}}  \right) + \frac{\partial}{\partial y} \left(
 f_{x} + y\, {\mathbf{C}}_{{\delta, \tau}} \right) \\
 &=   \frac{\partial}{\partial x} \left(  x\, {\mathbf{C}}_{{\delta, \tau}}   \right) + \frac{\partial}{\partial y}
  \left(  y\, {\mathbf{C}}_{{\delta, \tau}}   \right) =  \frac{2\tau}{{\delta}^{2}}.
\end{align*}
Hence the spacelike graph $z=g(x,y)$ over the same domain $\Omega'$ in $\mathbb{L}^{3}(M,\tau)$ has mean curvature $\tau$.

We will now deal with the last paragraph in the statement. Let $\zeta=\xi+i\eta$ be a local complex conformal coordinate on the graph $z=f(x,y)$ and consider the coordinate transformation $\phi:(x,y) \rightarrow  \left( \xi, \eta   \right) $  to obtain the  local isothermal parametrization  of the graph $z=f(x,y)$  in $\mathbb{E}^{3}(M,\tau)$ given by
\[
\widehat{F}=F \circ {\phi}^{-1} : \left(  \xi, \eta    \right) \mapsto (x,y) \mapsto (x,y, f(x,y)).
\]
Here, ${\phi}^{-1}$ denotes the  local inverse coordinate transformation
$\left( \xi, \eta   \right)  \mapsto (x,y)=\left(p \left( \xi, \eta   \right) , q \left( \xi, \eta   \right)  \right)$, so our goal is to prove that
 \[
  \widehat{G}=G \circ {\phi}^{-1} : \left( \xi, \eta   \right) \mapsto (x,y) \mapsto (x,y, g(x,y))
 \] 
is an isothermal parametrization of the twin graph $z=g(x,y)$ in $\mathbb{L}^{3}(M,H)$ and that the resulting conformal factor is $\frac{1}{\omega^2}$. Taking the complexified operator $\frac{\partial}{\partial \zeta}=  \frac{1}{2} \left( \frac{\partial}{\partial \xi} -i \frac{\partial}{\partial \eta} \right)$ and considering the $\mathbb C$-linear extensions of the metrics in $\mathbb E^3(M,\tau)$ and $\mathbb L^3(M,H)$, the proof will be finished once there have been proved
\[
 \left\langle\frac{\partial \widehat{G}}{\partial \zeta},  \frac{\partial \widehat{G}}{\partial \zeta} \right\rangle_{\mathbb L^3(M,H)}=0 \qquad \text{and} \qquad
{\left\langle \, \frac{\partial \widehat{G}}{\partial \zeta},  \, \overline{ \frac{\partial \widehat{G}}{\partial \zeta} } \, \right\rangle}_{\mathbb{L}^{3}(M,H)}
=\frac{1}{\omega^2} \left\langle\frac{\partial \widehat{F}}{\partial \zeta},\overline{\frac{\partial \widehat{F}}{\partial \zeta} } \, \right\rangle_{\mathbb{E}^{3}(M,\tau)}.\]
With respect to the global orthonormal frames $\{E_1,E_2,E_3\}$ in $\mathbb{E}^{3}(M,  \tau)$ and $\{L_1,L_2,L_3\}$ in $\mathbb{L}^3(M,H)$, given by~\eqref{eqn:frame} for $\epsilon=1$ and $\epsilon=-1$, respectively,
\begin{equation}\label{eqn:derivatives}
\begin{aligned}
       \frac{\partial \widehat{F}}{\partial \zeta}  
  &= \frac{\partial p}{\partial \zeta}   \frac{\partial  {F}}{\partial x}  
        +   \frac{\partial p}{\partial \zeta}   \frac{\partial  {F}}{\partial y} 
 = \frac{1}{  {\delta}  } \frac{\partial p}{\partial \zeta}  E_{1} +  \frac{1}{  {\delta}  }  \frac{\partial q}{\partial \zeta}    E_{2} +  \left( \alpha  \frac{\partial p}{\partial \zeta}  + \beta  \frac{\partial q}{\partial \zeta}   \right)   E_{3},\\
 \frac{\partial \widehat{G}}{\partial \zeta}  
 &=  \frac{\partial p}{\partial \zeta}   \frac{\partial G}{\partial x}  
        +   \frac{\partial p}{\partial \zeta}   \frac{\partial  G}{\partial y}
        =\frac{1}{  {\delta} }  \frac{\partial p}{\partial \zeta}   L_{1}  +   \frac{1}{  {\delta}  } \frac{\partial q}{\partial \zeta}   L_{2}  +  \left(\widetilde{\alpha}  \frac{\partial p}{\partial \zeta}  + \widetilde{\beta}  \frac{\partial q}{\partial \zeta}   \right)   L_{3} . 
\end{aligned}
\end{equation}
Using the twin relation $( \widetilde{\alpha} , \widetilde{\beta} ) = ( - \frac{ \beta  }{  \omega },  \frac{  \alpha }{ \omega} )$, it is easy to deduce following two equalities that will be useful in future computations:
\[
\frac{1 -   {\widetilde{\alpha}}^{2}  {\delta}^{2}  }{ {\delta}^{2}  } 
= \frac{1}{  {{\omega}}^{2}  } \left(   \frac{1}{  {\delta}^{2}  }  + {\alpha}^{2}  \right),
\qquad
\frac{1 - {  {\widetilde{\beta}}  }^{2}  {\delta}^{2}  }{ {\delta}^{2}  } 
= \frac{1}{  {{\omega}}^{2}  } \left(   \frac{1}{  {\delta}^{2}  }  + {\beta}^{2}  \right).
\]
Finally, by using~\eqref{eqn:derivatives}, we get
\begin{align*}
  {\left\langle \frac{\partial \widehat{G}}{\partial \zeta},  \frac{\partial \widehat{G}}{\partial \zeta} \right\rangle}_{\mathbb{L}^{3}(M, H) }    
   &=     \left( \frac{1}{{\delta} } \frac{\partial p}{\partial \zeta} \right)^{2} +  \left( \frac{1}{{\delta} }  \frac{\partial q}{\partial \zeta} \right)^{2} 
   -  \left(  \widetilde{\alpha}  \frac{\partial p}{\partial \zeta}  +\widetilde{\beta}  \frac{\partial q}{\partial \zeta}  \right)^{2}   \\ 
    &= \left(  \frac{1- {\widetilde{\alpha}  }^{2} {{\delta} }^{2} }{ {{\delta} }^{2}   }  \right) { \left( \frac{\partial p}{\partial \zeta} \right)}^{2} + 
     \left( \frac{1-{ \widetilde{\beta}  }^{2} {{\delta} }^{2}}{  {{\delta} }^{2}   } \right) {\left( \frac{\partial q}{\partial \zeta} \right) }^{2} - 2  \widetilde{\alpha}  \widetilde{\beta} \frac{\partial p}{\partial \zeta} \frac{\partial q}{\partial \zeta} \\ 
    &= \left( \frac{1 +  {\alpha}^2 {{\delta} }^{2} }{{\omega}^{2}  {{\delta} }^{2} }  \right) {\left( \frac{\partial p}{\partial \zeta} \right) }^{2} +  \left( \frac{1 +  {\beta}^2 {{\delta} }^{2} }{{\omega}^{2}  {{\delta} }^{2} } \right) {\left( \frac{\partial q}{\partial \zeta} \right) }^{2} - 2   \left( - \frac{ \alpha \beta}{ {\omega}^{2} } \right)  \frac{\partial p}{\partial \zeta} \frac{\partial q}{\partial \zeta} \\  
  &= \frac{1}{{\omega}^{2}}  \left[ \, \left( \frac{1}{{\delta} } \frac{\partial p}{\partial \zeta} \right)^{2} +  \left( \frac{1}{{\delta} } \frac{\partial q}{\partial \zeta} \right)^{2} 
  +  \left(   {\alpha}  \frac{\partial p}{\partial \zeta}  + {\beta}  \frac{\partial q}{\partial \zeta}  \right)^{2} \, \right]  \\ 
  &= \frac{1}{{\omega}^{2}}  {\left\langle \frac{\partial \widehat{F}}{\partial \zeta},  \frac{\partial \widehat{F}}{\partial \zeta} \right\rangle}_{\mathbb{E}^{3}(M,\tau) } = 0,
\end{align*} 
where, in the last line we used that $\zeta$ is conformal in $z=f(x,y)$, and
 \begin{align*} 
  {\left\langle \, \frac{\partial \widehat{G}}{\partial \zeta},  \, \overline{ \frac{\partial \widehat{G}}{\partial \zeta} } \, \right\rangle}_{\mathbb{L}^{3}(M,  H) } 
   &=     {\left\vert \frac{1}{{\delta} } \frac{\partial p}{\partial \zeta} \right\vert}^{2} +   {\left\vert \frac{1}{{\delta} }  \frac{\partial q}{\partial \zeta} \right\vert}^{2} 
   -  \left(  \widetilde{\alpha}  \frac{\partial p}{\partial \zeta}  +\widetilde{\beta}  \frac{\partial q}{\partial \zeta}  \right) 
   \left(  \widetilde{\alpha} \overline{  \frac{\partial p}{\partial \zeta} } + \widetilde{\beta} \overline{  \frac{\partial q}{\partial \zeta} } \right) \\ 
    &= \left(  \frac{1- {\widetilde{\alpha}  }^{2} {{\delta} }^{2} }{ {{\delta} }^{2}   }  \right) {  {\left\vert  \frac{\partial p}{\partial \zeta} \right\vert}}^{2} + 
     \left( \frac{1-{ \widetilde{\beta}  }^{2} {{\delta} }^{2}}{  {{\delta} }^{2}   } \right) {{\left\vert \frac{\partial q}{\partial \zeta} \right\vert} }^{2} -   \widetilde{\alpha}  \widetilde{\beta} \left(  \frac{\partial p}{\partial \zeta} \overline{ \frac{\partial q}{\partial \zeta} }
       +  \overline{   \frac{\partial p}{\partial \zeta}} \frac{\partial q}{\partial \zeta}  
     \right) \\ 
    &= \left( \frac{1 +  {\alpha}^2 {{\delta} }^{2} }{{\omega}^{2}  {{\delta} }^{2} }  \right) {\left\vert \frac{\partial p}{\partial \zeta} \right\vert }^{2} +  \left( \frac{1 +  {\beta}^2 {{\delta} }^{2} }{{\omega}^{2}  {{\delta} }^{2} } \right) {\left\vert \frac{\partial q}{\partial \zeta} \right\vert }^{2} - 2   \left( - \frac{ \alpha \beta}{ {\omega}^{2} } \right) \left(  \frac{\partial p}{\partial \zeta} \overline{ \frac{\partial q}{\partial \zeta} }
       +  \overline{   \frac{\partial p}{\partial \zeta}} \frac{\partial q}{\partial \zeta}  
     \right) \\  
  &= \frac{1}{{\omega}^{2}}  \left[ \,  {\left\vert \frac{1}{{\delta} } \frac{\partial p}{\partial \zeta} \right\vert}^{2} +  
  {\left\vert \frac{1}{{\delta} } \frac{\partial q}{\partial \zeta} \right\vert}^{2} 
  +  \left(   {\alpha}  \frac{\partial p}{\partial \zeta}  + {\beta}  \frac{\partial q}{\partial \zeta}  \right) 
   \left(   {\alpha} \overline{ \frac{\partial p}{\partial \zeta}}  + {\beta} \overline{ \frac{\partial q}{\partial \zeta} } \right) \, \right]  \\ 
  &= \frac{1}{{\omega}^{2}} {\left\langle \, \frac{\partial \widehat{F}}{\partial \zeta}, \, \overline{\frac{\partial \widehat{F}}{\partial \zeta} } \, \right\rangle}_{\mathbb{E}^{3}(M,  \tau) }.\qedhere
\end{align*}
\end{proof}

\begin{remark}[\textbf{Angle functions}]
The functions $\omega$ and $\widetilde\omega$ in the twin correspondence have an interesting geometric interpretation. Note that the upward-pointing unit normal vector fields ${\mathcal{N}}_{f}$ of the graph $z=f(x,y)$ in $\mathbb{E}^{3}(M,\tau)$, and $\mathcal N_g$ of the graph $z=g(x,y)$ in $\mathbb{L}^{3}(M,H)$ have the forms
\begin{align*}
{\mathcal{N}}_{f} &= - \frac{ \alpha  {\delta} }{\omega}   E_{1}  - \frac{ \beta  {\delta} }{\omega}   E_{2}   + \frac{1}{{\omega}}   E_{3},&
{\mathcal{N}}_{g} &= - \frac{ \widetilde\alpha  {\delta} }{\widetilde\omega}   L_{1}  - \frac{ \widetilde\beta  {\delta} }{\widetilde\omega}   L_{2}   + \frac{1}{{\widetilde\omega}}   L_{3}.
\end{align*}
The so-called \textit{angle functions} $u=\langle\mathcal N_f,E_3\rangle_{\mathbb E^3(M,\tau)}$ and $\widetilde u=\langle\mathcal N_g,L_3\rangle_{\mathbb L^3(M,\tau)}$ are nothing but $u=\frac{1}{\omega}$ and $\widetilde u=\frac{1}{\widetilde\omega}$. In particular, twin surfaces have reciprocal angle functions.
\end{remark}

\begin{corollary}[\textbf{Prescribed mean curvature as zero mean curvature}] 
Let $H: \mathbb{R}^{2} \rightarrow \mathbb{R}$ be a smooth function, and let $\Omega \subset  \mathbb{R}^{2}$ be open and simply connected.
\begin{itemize}
\item[\textbf{(a)}]  
 There exists a twin correspondence between graphs over $\Omega$ with prescribed mean curvature $H$ in Euclidean space $\mathbb{R}^{3}$ and maximal graphs over $\Omega$ in the generalized Heisenberg spacetime ${\mathrm{Nil}}^{3}_{1}(H)=\mathbb{L}^{3}({\mathbb{E}}^{2},H)$.
\item[\textbf{(b)}]  There exists a twin correspondence between spacelike graphs over $\Omega$ with prescribed mean curvature $H$ in Lorentz--Minkowski space $\mathbb{L}^{3}$ and minimal graphs defined over
$\Omega$  in the generalized Heisenberg space ${\mathrm{Nil}}^{3} (H)=\mathbb{E}^{3}({\mathbb{E}}^{2},H)$.
\end{itemize}
\end{corollary}

\begin{corollary}[\textbf{Twin correspondence in the BCV spaces, \cite[Theorem 2]{Lee11a}}]\label{coro:bcv}
Given constants $\kappa, \tau, H \in {\mathbb{R}}$, there exists a twin correspondence between 
graphs with mean curvature $H$ in 
the Riemannian BCV space ${\mathbb {E}}^{3}(\kappa,\tau)$ and 
spacelike graphs with mean curvature $\tau$ in the BCV spacetime ${\mathbb{L}}^{3}(\kappa, H)$.
\end{corollary}

\begin{example}[\textbf{Twin surfaces of helicoidal surfaces with constant mean curvature $\tau$ in ${\mathbb{L}}^{3}$ are catenoids in ${\mathrm{Nil}}^{3} (\tau)$}]
 Let $\lambda \geq 0$ and consider the spacelike helicoidal surface with constant mean curvature $\tau$ in ${\mathbb{L}}^{3}=\mathbb{L}^{3}(0,0)$,
 \[
     z=g(x,y)= \lambda \arctan{\left( \frac{y}{x} \right)} + \tau \,  h \left( \sqrt{x^2 +y^2} \right),
 \]
where the function $h:  (\lambda, \infty) \rightarrow \mathbb{R}$ satisfies the ODE
 \[h'(t) =  \sqrt{ \frac{ t^{2} - {\lambda}^{2} }{ {\tau}^{2} t^2 + 1 } }.\]  
Its twin surface in ${\mathrm{Nil}}^{3} \left(  \tau\right)=\mathbb{E}^{3}(0, \tau)$
is the half catenoid $z=f(x,y)$ defined over the domain $\sqrt{x^2 +y^2} >\lambda$. It is a rotationally invariant  minimal surface of the form
 \[z=f(x,y)=\lambda  \rho\left( \sqrt{x^2 +y^2} \right),\] 
where the one-variable function $\rho:  (\lambda, \infty) \rightarrow \mathbb{R}$ satisfies the ODE
 \[\rho'(t) =\sqrt{ \frac{ {\tau}^{2} t^2 + 1 }{ t^{2} - {\lambda}^{2} }}.\]
\end{example} 

\section{Complete spacelike surfaces}  \label{MTapp}

Entire graphs in $\mathbb E(M,\tau)$ are complete when the base surface $M$ is complete. This assertion fails to be true in the Lorentzian case in general, as shown by the examples constructed by Albujer~\cite{Alb08} in the \textit{Robertson-Walker} spacetime $\mathbb H^2(-1)\times \mathbb{R}$. We will begin by proving that complete spacelike surfaces in $\mathbb{L}^{3}(M,\tau)$ are entire graphs.

Throughout this section, we will deal with a non-compact simply connected surface $M$, conformally parametrized as $M=(\Omega,g=\delta^{-2}(\df x^2+\df y^2))$, where $\delta\in\mathcal C^\infty(\Omega)$ is positive, and $\Omega\subseteq\mathbb R^2$ is open and star-shaped with respect to the origin.
 
\begin{lemma} [{\textbf{Covering map lemma}, \cite[Ch.~VIII, Lemma 8.1]{KN}}] \label{coverMAP}
Let $\phi$ be a map from a connected complete Riemmanian manifold ${\mathcal{R}}_{1}$ onto another connected Riemmanian manifold ${\mathcal{R}}_{2}$ of the same dimension. If the map $\phi$ is distance non-decreasing,  then $\phi$ is a covering map, and ${\mathcal{R}}_{2}$ is also complete.
\end{lemma}

\begin{lemma}[\textbf{Complete spacelike surfaces in $\mathbb{L}^{3}(M,\tau)$ are entire graphs}]\label{comentire}
If there exists a complete spacelike surface $\Sigma$ in $\mathbb{L}^{3}(M,\tau)$, then $\Sigma$ is an entire graph and $M$ is also complete.
\end{lemma}

\begin{proof}  The same ideas as in \cite[Lemma 3.1]{AA09} and \cite[Proposition 3.3]{ARS95} work here. We begin with a complete spacelike surface $\Sigma \subset \mathbb L^3(M,\tau)$. From~\eqref{Rmetric}, we get the following upper bound for the metric in $\mathbb L^3(M,\tau)$:
 \begin{equation*} 
\frac{1}{{\delta}^{2}} \left(\df x^2 + \df y^2\right)
  - {\left(\df z - \mathbf{C}_{\delta,\tau}(x,y) \left( y\,\df x - x\,\df y  \right)
   \right)}^{2}  \leq \frac{1}{{\delta}^{2}} \left(\df x^2 + \df y^2\right).
 \end{equation*}
Thus the projection $\pi:(x,y,z) \in \Omega\times\mathbb R \to (x,y) \in \Omega$ from the  ambient space $\mathbb L^3(M,\tau)$ to the base  $M=(\Omega, \delta^{-2}(\df x^2 + \df y^2))$ is a distance non-decreasing map. 

Since the induced metric in the spacelike surface $\Sigma$ is complete, Lemma \ref{coverMAP} applied to $\pi_{|\Sigma}: \Sigma \to \Omega$ ensures that $\pi_{|\Sigma}$ is a covering map and $\Sigma$ is complete. As the domain $\Omega$ is simply connected, the covering map  $\pi_{|\Sigma}: \Sigma \rightarrow \Omega$ must be a global diffeomorphism, so $\Sigma$ is an entire graph.
\end{proof}
 
In 1976, Cheng and Yau \cite{CY76} proved the remarkable result that any entire spacelike graph with constant mean curvature in flat Lorentz--Minkowski space (in particular, in $\mathbb{L}^{3}=\mathbb{L}^{3}(0, 0)$) is complete. We will now show that a similar conclusion is true for maximal surfaces in the 3-dimensional Lorentzian hyperbolic space. The proof relies on combining the twin correspondence with the Daniel sister correspondence~\cite{Dan07}, which gives an isometric correspondence between entire minimal graphs in Heisenberg space and entire graphs with mean curvature $\frac{1}{2}$ in $\mathbb H^2\times\mathbb R$.
 
\begin{theorem}[\textbf{Complete maximal surfaces in anti--de Sitter spacetime}] \label{CY} 
Given a constant $\kappa<0$, let $\Sigma$ be a maximal surface in the anti--de Sitter spacetime $\mathbb{L}^{3}(\kappa, \frac{1}{2}\sqrt{-\kappa})$.  The following two statements are equivalent:
\begin{itemize}
\item[\textbf{(a)}] The surface $\Sigma$ is complete.
\item[\textbf{(b)}] The surface  $\Sigma$ is an entire graph over the whole hyperbolic plane ${\mathbb{H}}^{2}(\kappa)$.
\end{itemize}
\end{theorem}

\begin{proof}
From Lemma \ref{comentire}, it is clear that \textbf{(a)}  $\Rightarrow$ \textbf{(b)}, so we will focus on \textbf{(b)}  $\Rightarrow$ \textbf{(a)}. By homothetically rescaling the metric of $\mathbb L^3(\kappa,\frac{1}{2}\sqrt{-\kappa})$, without loss of generality, we assume that $\kappa=-1$.

Let $\Sigma$ be an entire maximal graph in $\mathbb{L}^{3}(-1,\frac{1}{2})$, and consider its twin entire graph $\Sigma^*$  with constant mean curvature $\frac{1}{2}$ in ${\mathbb{H}}^{2} \times {\mathbb{R}}=\mathbb{E}^{3}(-1, 0)$. Since $\Sigma^*$ is simply connected, we can take its sister minimal surface $\hat\Sigma^*$ in the Heisenberg group $\mathrm{Nil}^3=\mathbb{E}^3(0,\frac{1}{2})$ via the Daniel correspondence~\cite{Dan07}. From~\cite[Corollary 3.3]{DH09} we deduce that $\hat\Sigma^*$ is also an entire graph over $\mathbb E^2$, so we can employ the twin correspondence again to associate an entire spacelike graph $\hat\Sigma$ in $\mathbb L^3(0,0)$ with constant mean curvature $\frac{1}{2}$:
\[\xymatrix{\Sigma\subset\mathbb L^3(-1,\tfrac{1}{2})\ar@{<->}[d]\ar@{<-->}[rr]&&\hat\Sigma\subset\mathbb{L}^3(0,0)\ar@{<->}[d]\\\Sigma^*\subset\mathbb{E}^3(-1,0)\ar@{<->}[rr]&&{\hat\Sigma}^*\subset\mathbb{E}^3(0,\frac{1}{2})}\]
Since Daniel's isometric correspondence between $\Sigma^*$ and $\hat\Sigma^*$ preserves the angle function, and Theorem~\ref{TCb} implies that the metrics of the Lorentzian graphs are conformal to the those of the corresponding Riemannian space with conformal factor the square of the angle function, we deduce that $\Sigma$ and $\hat\Sigma$ are isometric surfaces. Since $\hat\Sigma$ is complete by Cheng and Yau's result \cite{CY76}, so is $\Sigma$.
\end{proof}

\begin{remark}
The proof of Theorem~\ref{CY} shows that the moduli space of the complete maximal surfaces in the anti--de Sitter spacetime is large. 
According to Fern\'{a}ndez--Mira \cite[Theorem 1 and Proposition 14]{FM09} or Cartier--Hauswirth \cite[Theorem 3.9]{CH12}, there exist many entire graphs with constant mean curvature $\frac{1}{2}$ in the product space ${\mathbb{H}}^{2} \times \mathbb{R}=\mathbb{E}^{3}(-1, 0)$, which give rise to entire graphs with constant mean curvature $\frac{1}{2}\sqrt{-\kappa}$ in ${\mathbb{H}}^{2}(\kappa) \times \mathbb{R}=\mathbb{E}^{3}(\kappa,0)$ by rescaling the metric. By the twin correspondence, these surfaces correspond to entire maximal graphs in $\mathbb L^3(\kappa,\frac{1}{2}\sqrt{-\kappa})$, which are complete by the equivalence between \textbf{(a)}  and \textbf{(b)} in Theorem~\ref{CY}.
\end{remark}

\begin{remark} 
Bonsante and Schlenker \cite{BS10} used the geometry of maximal surfaces in the anti--de Sitter spacetime to give a variant of Schoen's conjecture on the universal Teichm\"{u}ller space. The proof of Theorem~\ref{CY} gives geometrical equivalences between the following entire graphs in different spaces:
\begin{itemize}
\item[\textbf{(a)}] entire maximal spacelike graphs (defined over the hyperbolic plane ${\mathbb{H}}^{2}$) in the anti--de Sitter spacetime $\mathbb{L}^{3}(-1, \frac{1}{2})$. 
\item[\textbf{(b)}] entire mean curvature $\frac{1}{2}$ graphs (defined over the hyperbolic plane ${\mathbb{H}}^{2}$) in the Riemannian product space ${\mathbb{H}}^{2}\times {\mathbb{R}}$.
\item[\textbf{(c)}] entire mean curvature $\frac{1}{2}$ spacelike graphs (defined over the Euclidean plane ${\mathbb{E}}^{2}$) in the Lorentz--Minkowski space ${\mathbb{L}}^{3}$.
\item[\textbf{(d)}] entire minimal graphs (defined over the Euclidean plane ${\mathbb{E}}^{2}$) in the Heisenberg space  $\mathrm{Nil}^3(\frac{1}{2})=\mathbb{E}^{3}(0, \frac{1}{2})$. 
\end{itemize}
\end{remark}

In order to give a sharp non-existence result for complete spacelike surfaces in Lorentzian Killing submersions, we introduce the Cheeger isoperimetric constant.

 \begin{definition}\label{defi:cheeger}
 The Cheeger constant of a non-compact Riemannian surface $M$ without boundary is defined as 
\begin{equation}\label{cheeger}
\mathrm{Ch}(M)=\inf\left\{\frac{\Long(\partial D)}{\Area(D)}:D\subset M\text{ open and regular}\right\} 
\geq 0.
\end{equation}
Here, an open subset $D\subset M$ is \textit{regular} if it is relatively compact and its boundary is a smooth curve so the quotient in \eqref{cheeger} makes sense. 
\end{definition}

\begin{theorem} \label{NON}
Let $M$ be a non-compact simply connected surface. 
\begin{itemize}
 \item[\textbf{(a)}] Given $H\in \mathcal C^\infty(M)$ such that $\inf_M|H|>\frac{1}{2}\mathrm{Ch}(M)$, the space $\mathbb E^3(M,\tau)$ admits no entire graphs with mean curvature $H$ for any $\tau\in\mathcal C^\infty(M)$.
 \item[\textbf{(b)}] Given $\tau\in\mathcal C^\infty(M)$ such that $\inf_M|\tau|>\frac{1}{2}\mathrm{Ch}(M)$, the spacetime $\mathbb L^3(M,\tau)$ admits neither complete spacelike surfaces nor entire spacelike graphs.
\end{itemize}
\end{theorem}

\begin{proof}
We will use a classical argument \cite{ER,He55, Pen12, Sa89} originally due to Heinz to obtain item \textbf{(a)}. Aiming at a contradiction, suppose that such an entire graph exists, and assume first that $H>0$ (the case $H<0$ will be treated later). Its mean curvature $H$ admits the expression
\begin{equation}\label{divergence-omega}
2H=\mathrm{div}_M\left(\frac{G}{\sqrt{1+{\|G\|}^2_M}}\right),
\end{equation}
for some vector field $G$ on $M$ as in~\eqref{eq:H}. Letting $H_0=\inf_M(H)$ and integrating~\eqref{divergence-omega} over an open regular domain $D\subset M$, we get
\begin{align*}
             2H_0\,\Area(D) 
 \leq \int_{D}   \mathrm{div}_{M}\left(\frac{G}{\sqrt{1+{\|G\|}_{M}^{2}}}\right)   
 = \int_{\partial D}\frac{ \langle G,\eta\rangle }{\sqrt{1+{\|G\|}_{M}^{2}}}\leq \Long(\partial D),
\end{align*}
where $\eta$ denotes the outer unit conormal vector field to $D$ along its boundary and we have used the divergence formula and Cauchy-Schwarz inequality. As this is valid for all open regular domains, we deduce that 
\[\textstyle H_0=\inf_M(H)=\inf_M|H|<\tfrac{1}{2}\mathrm{Ch}(M),\] contradicting the hypothesis in the statement. If $H<0$, then the argument above can be adapted by replacing $G$ by $-G$, to get that $-2H_0\Area(D)\leq\Long(\partial D)$, so $-H_0=\inf_M|H|<\frac{1}{2}\mathrm{Ch}(M)$ and we also get a contradiction.

In order to prove item \textbf{(b)}, we will reason by contradiction again: if there existed such a complete spacelike surface $\Sigma$, then $\Sigma$ would be an entire graph by Lemma~\ref{comentire} so its twin surface $\widetilde\Sigma$ would be an entire graph in $\mathbb E^3(M,H)$, where $H$ denotes the mean curvature of $\Sigma$. The mean curvature of $\widetilde\Sigma$ would be $\tau$, satisfying $\inf_M|\tau|>\frac{1}{2}\mathrm{Ch}(M)$ and contradicting item \textbf{(a)}.
\end{proof}

\begin{corollary}\label{coro:infimum}
Let $M$ be a complete non-compact simply connected surface and let $c=\inf\{K(p):p\in M\}\leq 0$, where $K$ denotes the Gaussian curvature of $M$. 
\begin{itemize}
 \item[\textbf{(a)}] Given $H\in \mathcal C^\infty(M)$ such that $\inf_M|H|>\frac{1}{2}\sqrt{-c}$, the space $\mathbb E^3(M,\tau)$ admits no entire graphs with prescribed mean curvature $H$ for any $\tau\in\mathcal C^\infty(M)$.
 \item[\textbf{(b)}] Given $\tau\in\mathcal C^\infty(M)$ such that $\inf_M|\tau|>\frac{1}{2}\sqrt{-c}$, the spacetime $\mathbb L^3(M,\tau)$ admits neither complete spacelike surfaces nor entire spacelike graphs.
\end{itemize}
\end{corollary}

\begin{proof}
 The estimate in \cite[Lemma 4.1]{ER} gives $\mathrm{Ch}(M)\leq\sqrt{-c}$. Then, the statements \textbf{(a)} and \textbf{(b)} immediately follow from Theorem~\ref{NON}.
\end{proof}

Observe that Theorem~\ref{CY} (for $c<0$) and the classical classification result by Calabi~\cite{Cal70} (for $c=0$) show that the lower bound $\frac{1}{2}\sqrt{-c}$ in Corollary~\ref{coro:infimum} is sharp. 

To conclude, observe that Theorem~\ref{NON} also gives information about causality in $\mathbb L^3(M,\tau)$ spaces when we look at them as spacetimes. A spacetime $\mathcal{L}$ is said to be \textit{distinguishing} when any two different points $p,q\in \mathcal{L}$ have different future or past cones. Equivalently, if for any $p\in \mathcal{L}$ and any neighborhood $U$ of $p$, there exists a neighborhood $V\subset U$ of $p$, such that causal (i.e., non-spacelike) curves starting at $p$ and leaving $V$ never enter $V$ again. 

\begin{corollary}
The spacetime $\mathbb L^3(M,\tau)$ is not distinguishing when the bundle curvature $\tau$ satisfies $\inf_M|\tau|>\frac{1}{2}\mathrm{Ch}(M)$.
\end{corollary}

\begin{proof} 
Distinguishing spacetimes with a complete timelike Killing vector field admit a Riemannian submersion structure whose fibers are the integral curves of the Killing vector field, and they also admit complete spacelike surfaces~\cite{JS08}. Hence the corollary follows from Theorem~\ref{NON}. 
\end{proof}

\noindent\textbf{Acknowledgement.}  This work was initiated while the first author was visiting the University of Granada during the 2011-2012 academic year. He would like to thank the department of Geometry and Topology and his mentor Joaqu\'{i}n P\'{e}rez for their warm hospitality and for providing an excellent working environment. The authors would like to thank Miguel S\'{a}nchez and Miguel \'{A}ngel Javaloyes for their useful comments, as well as the anonymous referees for their extremely meticulous reports. The first author was supported by the National Research Foundation of Korea Grant funded by the Korean Government (Ministry of Education, Science and Technology) [NRF-2011-357-C00007].  The second author was supported by the Spanish MCyT-Feder Research Project MTM2014-52368-P, and by the EPSRC Grant No.~EP/M024512/1.

\end{document}